\documentclass[12pt]{amsart}

\usepackage{color}
\usepackage{hyperref}
\usepackage{graphicx}
\graphicspath{ {./images/} }
\usepackage{amsfonts,amssymb,amscd,amsmath,enumerate,verbatim}
\usepackage[normalem]{ulem}
\usepackage{mwe}
\usepackage{wrapfig}
\usepackage{pdflscape}
\newcommand{\Ap}{\text{ Ap}}
\newcommand{\G}{\Gamma}

\newtheorem{theorem}{Theorem}[section]

\newtheorem{corollary}[theorem]{Corollary}
\newtheorem*{corollary*}{Corollary}
\newtheorem{proposition}[theorem]{Proposition}
\newtheorem{example}[theorem]{Example}

\newtheorem{remark}{Remark}
\newtheorem{definition}[theorem]{Definition}

\newtheorem{question}{Question}

\newcommand{\bs}{\backslash}

\def\N{\mathbb{N}}

\def\Z{\mathbb{Z}}
\def\a{{\bf a}}

\newcommand{\la}{\langle}
\newcommand{\ra}{\rangle}

\title[Principal Matrices of Certain Numerical Semigroups]
{A class of Numerical Semigroups defined by Kunz and Waldi - their Principal Matrices and structure}
      
\author[S.~Singh]{Srishti Singh}

\author[H.~Srinivasan]{Hema Srinivasan}

\address{Department of Mathematics, 
         University of Missouri, 
         Columbia MO 65211, USA.}
    
\date{\today}

\begin{document}

\begin{abstract}

In this paper, we explore a class of numerical semigroups initiated by Kunz and Waldi containing two coprime numbers $p<q$, which we call KW semigroups. We characterize KW numerical semigroups by their principal matrices. We present a necessary and sufficient criterion for a matrix to be the principal matrix of a KW semigroup.  An explicit description of the minimal resolutions of numerical semigroups in the same class with small embedding dimensions $3$ and $4$ is given. We give a generalization of this notion to three dimensions using lattice paths under a plane and present some preliminary results and questions. 
\end{abstract}

\thanks{MSC Primary: {13c05}, {13c70}; Secondary: {13D02}.}


\maketitle

\section{Introduction}

A numerical semigroup $\la \a \ra$ is a submonoid of $\N$  minimally generated by  $\a = \{a_1, \ldots a_n\} \subseteq \N$ where  $\gcd (a_1, \ldots, a_n)= 1$. Such a semigroup will contain all but finitely many positive integers, called the gaps. The largest gap is called the Frobenius number, $F(\a)$.

This paper is inspired by recent remarkable works of Kunz and Waldi in \cite{KUNZ2017397} where they build numerical semigroups of the same multiplicity $p$ by filling in gaps. Any semigroup of embedding dimension $2$ is symmetric, i.e., the number of gaps is exactly half the number of elements less than the Frobenius number.  Thus, a semigroup generated by two relatively prime positive integers $p<q$, is symmetric and of multiplicity $p$.   Kunz and Waldi build extensions of these semigroups by filling in gaps larger than $p$  to create semigroups of higher embedding dimensions with the same multiplicity $p$.  All the gaps of $\la p,q \ra$ are of the form $pq-xp-yq$ for some positive integers $x, y \in \N$, with $pq-p-q$ being the largest, the Frobenius number.  

In \cite{KUNZ2017397}, Kunz and Waldi study numerical semigroups of the form $\la p,q, h_1, \ldots h_{n-2}\ra$ where $p<q$ are relatively prime positive integers.    Further, they analyze a special class of semigroups with embedding dimension $n$, which we denote by $KW(p,q)$, by filling in gaps where $x\le p/2, y\le q/2$.   In particular, they prove a structure theorem for this class: 
\begin{theorem}\label{KWmain} (Corollary $3.1$, Theorem (Appendix) \cite{KUNZ2017397})
    Let $A = \la p, q, h_1, \ldots h_{n-2}\ra$ with 
$h_i = pq-x_ip-y_iq,$ for some $0< 2x_i \le q, 0< 2y_i \le p$, $1 \leq i \leq n-2$.  Then the semigroup ring $k[A] $ is of type $n-1$ and its relation ideal is generated by $\binom{n}{2}$ elements. 
\end{theorem}

In this paper,  we begin by studying the principal matrices of these semigroups. A ``principal matrix" is an $n\times n$ integer matrix associated with a numerical semigroup of embedding dimension $n$.  If $\a= \{ a_1, \ldots, a_n\}$ minimally generates a numerical semigroup, then its principal matrix  $P(\a)= ( a_{ij})$ is an $n\times n$ matrix,  $ a_{ij} \ge 0, i\neq j$, where $-a_{ii}$ is the smallest positive integer such that $a_{ii}x_i+\sum_{i\neq j} a_{ij}x_j=0$. The diagonal entries of $P(\a)$ are  uniquely determined for a given $\a$ but the matrix itself need not be unique. 

Indeed for a general numerical semigroup $\la \a \ra$, the rank of $P(\a)\le n-1$ and when it is $n-1$, one can recover $\a$ from the $n-1$ order minors of any of the $n-1$ rows. Further, the rank of $P(\a) \ge n/2$, (\cite{MR4428293}).  As such, this matrix holds vital information about the semigroup itself. 

For instance, a criterion is established for Gorenstein monomial curves using principal matrices in \cite{Br1979} and \cite{Gimenez_2014}. Additionally, \cite{MR4428293} presents a characterization for principal matrices of numerical semigroups that are gluing of numerical semigroups.  
 
In this paper,  we study the principal matrices of the semigroups in $KW(p,q)$. We establish a necessary and sufficient criterion for a matrix to be the principal matrix of a numerical semigroup within Kunz and Waldi's class in Theorem \ref{main}. As a consequence, we get 

\begin{corollary}\label{maincor}
Let $3\le p<q$ be prime numbers.  Then the principal matrix of  a numerical semigroup $A =\la p , q, h_1, \ldots, h_{n-2} \ra \in KW((p,q))$  is of the form 
$$P(A)=  T(\alpha, \beta)=
    \begin{bmatrix}
    - \left(\frac{q + \alpha_{n-2}}{2} \right) & \frac{p-\beta_{n-2}}{2} & 0 & 0 & \cdots & 0 & 1 \\
    \frac{q- \alpha_1}{2} & - \left(\frac{p+\beta_1}{2} \right) & 1 & 0 & \cdots & 0 & 0 \\
    \alpha_1 & \beta_1 & -2 & 0 & \cdots & 0 & 0 \\
    \alpha_2 & \beta_2 & 0 & -2 & \cdots & 0 & 0 \\
    \vdots & \vdots & \vdots & \vdots & \ddots & \vdots & \vdots \\
    \alpha_{n-2} & \beta_{n-2} & 0 & 0 & \cdots & 0 & -2 
    \end{bmatrix},$$ for some odd positive integers $\alpha_i, \beta_i$ satisfying $q> \alpha_1 \ge \ldots \ge \alpha_{n-2}\ge 0$ and $0< \beta _1\le \cdots \leq \beta_{n-2}<p$.  Conversely, any $n\times n$ matrix of the form $P(A)= T(\alpha, \beta)$ above is a principal matrix of a $KW(p,q)$ numerical semigroup. 
\end{corollary}

As an interesting fact, we observe that the defining ideals of semigruop rings $k[H]$ for $H \in KW(p,q)$ can be written as sums of $2\times 2$ minors of $2+ \binom{n-2}{2}$ matrices of size $2\times 3$.  

The paper is organized as follows: In section $2$, we define the notion of principal matrices, the class $KW(p,q)$ of numerical semigroups and some notation. The proof of our main result (Theorem $\ref{main}$) is in section $3$. In section $4$, various auxiliary results pertaining to the minimal resolutions and relation ideal of semigroup rings of $KW(p,q)$ are given. We end this paper by extending the class of KW-semigroups in section $5$, and conclude with some open questions. 
 
\section{Preliminaries}

If a set of relatively prime positive integers $\mathbf{a} = \{a_1,...,a_n\}$   minimally generates a semigroup $\la \mathbf{a} \ra$, then, there are equations $$c_i a_i = \sum_{j\neq i, 1\le j\le n}a_{ij} a_i, \hskip .1truein 1\leq i \leq n.$$ In the associated semigroup ring    $k[\mathbf{a}] = k[t^{a_1}, \ldots, t^{a_n}] =k[x_1, \ldots, x_n]/I_{\mathbf{a}}$, $$f_i = x^{c_i}-\prod_{j\neq i, 1\le j\le n} x_j^{a_{ij}}$$ are among a set of minimal generators of $I_{\mathbf{a}}$.  Further, $\sqrt{(f_1, \ldots, f_n)} = \sqrt{ I_{\mathbf {a}}}$.  These $f_i$ are therefore called critical binomials of the numerical semigroup  $\la \mathbf{a} \ra$ or the binomial toric ideal $I_{\mathbf{a}}$.  This is also contained in the following $n\times n$ matrix, which is called the Principal Matrix of the numerical semigroup. 
\begin{definition}(Definition 1, \cite{MR4428293}) 
    Let $\mathbf{a} = \{a_1,...,a_n\}$ be a set of positive integers minimally generating a semigroup $\la \mathbf{a} \ra$. $A= 
\begin{bmatrix}
    -c_1 & a_{12} & \cdots & a_{1n} \\
    a_{21} & -c_2 & \cdots & a_{2n} \\
    \vdots & \vdots & \ddots & a_{2n} \\
    a_{n1} & a_{n2} & \cdots & -c_n
\end{bmatrix}$ is called a \textit{principal matrix} of $\la \mathbf{a} \ra$ if $A \mathbf{a}=0$, and $c_i$ is the smallest positive integer such that $c_ia_i \in \la \mathbf{a} - \{a_i\} \ra$ for all $1 \leq i \leq n$. 
\end{definition}

Although the diagonal entries $-c_{i}$ are uniquely determined, $a_{ij}$ are not always unique. The sequence of positive integers $\mathbf{a}$ can be recovered from a given principal matrix $A$ of rank $n-1$ by factoring out the greatest common divisor of the entries of any nonzero column of the adjoint of A and taking its absolute value.

\begin{example}
    The principal matrix of the numerical semigroup $\a= \la 5, 7, 11, 13\ra$ is $$P(\a) = \begin{pmatrix} -4 &1& 0&1\\
2&-3&1&0\\
3&1&-2&0\\
1&3&0&-2\\
\end{pmatrix}$$

We can recover $\a$ simply from the $3\times 3$ minors of the first three rows.  
\end{example}
In this paper, we study the principal matrices of numerical semigroups in $KW(p,q)$, constructed in \cite{article} as follows:

\begin{definition}
Let $p,q \in \N$ be relatively prime with $3 \leq p < q$.  The set of Kunz-Waldi Semigroups associated to $p<q$, denoted by   $KW(p,q)$, is  the set of all numerical semigroups $H$ with $\la p,q \ra \subset H \subset \la p,q,r \ra$, where 
$r=
\begin{cases}
\frac{p}{2}, & p \text{ even} \\
\frac{q}{2}, & q \text{ even} \\
\frac{p+q}{2}, & p \text{ and } q \text{ odd}. 
\end{cases}$\\
\end{definition}
 All $H \in KW(p,q)$ are in one-to-one correspondence to the lattice paths in the rectangle $R \subseteq \mathbb R^2$ with the corners $(0,0), (0,p'-1),(q'-1,p'-1),$ and $(q'-1,0)$, where $p'=\left\lfloor\frac{p}{2}\right\rfloor$ and $q'=\left\lfloor\frac{q}{2}\right\rfloor$. The only $H \in KW(p,q)$ with embedding dimension $e(H)=2$ are $\la p, q \ra, \la p/2,q \ra$, and $\la p, q/2 \ra$. For any other $H \in KW(p,q)$, the minimal generating set is $\{p,q,h_1,...,h_{e(H)-2}\}$, where $h_i = pq-x_ip-y_iq$, if $(x_i,y_i)$ for $1 \leq i \leq e(H)-2$ are the corners of the lattice path defining $H$. This gives a sufficient condition for $H$ to be in $KW(p,q)$ and we incorporate it in the definition: 

$H \in KW(p,q)$ if 
\begin{equation}\label{KWcondition}
    2x_i \leq q \quad \text{ and } \quad 2y_i \leq p \quad \quad \forall \; \; 1 \leq i \leq e(H)-2
\end{equation}

\begin{remark}
    Thus, $H = \la p,q,h_1, \ldots, h_{n-2} \ra \in KW(p,q)$ if and only if $h_i = pq-x_ip-y_iq$ satisfying
     $$2x_i \leq q \quad \text{ and } \quad 2y_i \leq p \quad \quad \forall \; \; 1 \leq i \leq e(H)-2,$$  $0<x_1<\ldots<x_{n-2}$, and $ y_1> \ldots >y_{n-2}>0$.
     \end{remark}

\section{Principal Matrix of KW semigroups}

In this section, we prove one of our main theorems \eqref{main}, which provides an explicit structure of the principal matrix for any $H \in KW(p,q)$.

\begin{theorem}\label{main} 
 Let $p<q$ be two relatively prime positive integers.  Choose positive integers $q > \alpha_1 > \alpha_2 > ... > \alpha_{n-2} \geq0$, of the same parity and $0 \leq \beta_1 < \beta_2 < ... < \beta_{n-2} < p$ of the same parity.   Let  $h_i = \frac{p \alpha_i + q \beta_i}{2}$, $1 \leq i \leq n-2$.

Then the matrix 
\begin{equation} \label{principalmatrix}
 T:=T(\alpha, \beta)=
    \begin{bmatrix}
    - \left(\frac{q + \alpha_{n-2}}{2} \right) & \frac{p-\beta_{n-2}}{2} & 0 & 0 & \cdots  & 1 \\
    \frac{q- \alpha_1}{2} & - \left(\frac{p+\beta_1}{2} \right) & 1 & 0 & \cdots &  0 \\
    \alpha_1 & \beta_1 & -2 & 0 & \cdots & 0 \\
    \alpha_2 & \beta_2 & 0 & -2 & \cdots & 0 \\
    \vdots & \vdots & \vdots & \vdots & \ddots & \vdots \\
    \alpha_{n-2} & \beta_{n-2} & 0 & 0 & \cdots  & -2 
    \end{bmatrix}
\end{equation}

    is a principal matrix for the semigroup $H= \{ p, q, h_1,...,h_{n-2} \}$ if and only if $H \in KW(p,q)$ provided  any of the following is true:
    \begin{enumerate}
        \item[(i)] $p$ and $q$ are both odd
        \item[(ii)] $p$ is even and $2h_1 \neq p\alpha_1$
        \item[(iii)] $p$ is even, $2h_1 = p\alpha_1$ but $q\le 2\alpha_1-\alpha_{n-2}$
        \item[(iv)] $q$ is even and $2h_{n-2} \neq q\beta_{n-2}$
        \item[(v)] $q$ is even and $p\le 2\beta_{n-2} - \beta_1$.
    \end{enumerate}

    If none of the above conditions is true and 
    if $p$ is even, then the principal matrix of $H$ is 
     $$
        \begin{bmatrix}
    - \alpha_1 & 0 & 2 & 0 & \cdots & 0  \\
    \frac{q- \alpha_1}{2} & - \left(\frac{p+\beta_1}{2} \right) & 1 & 0 & \cdots & 0  \\
    \alpha_1 & \beta_1 & -2 & 0 & \cdots &  0 \\
    \alpha_2 & \beta_2 & 0 & -2 & \cdots  & 0 \\
    \vdots & \vdots & \vdots & \vdots & \ddots & \vdots \\
    \alpha_{n-2} & \beta_{n-2} & 0 & 0 & \cdots & -2 
    \end{bmatrix}$$  and 
    if $q$ is even, it is $$  \begin{bmatrix}
   - \left(\frac{q + \alpha_{n-2}}{2} \right) & \frac{p-\beta_{n-2}}{2} & 0 & 0 & \cdots  & 1 \\
   0 & - \beta_{n-2} & 0 & 0 & \cdots & 2  \\
    \alpha_1 & \beta_1 & -2 & 0 & \cdots &  0 \\
    \alpha_2 & \beta_2 & 0 & -2 & \cdots  & 0 \\
    \vdots & \vdots & \vdots & \vdots & \ddots & \vdots \\
    \alpha_{n-2} & \beta_{n-2} & 0 & 0 & \cdots & -2 
    \end{bmatrix}.$$ 
\newcommand{\an}{\alpha_{n-2}}
\newcommand{\bn}{\beta_{n-2}}
\end{theorem}

\begin{proof}
First, suppose that one of the conditions $(i)$-$(v)$ are satisfied. 
Set $\mathbf h = \{p,q,h_1,...,h_{n-2}\}$. By construction,  $h_i$ is a minimal generator and $2h_i = p\alpha_i+ q\beta_i$. 
Next, we note that  
\begin{equation} \label{c11}
p\left( \frac{q+\alpha_{n-2}}{2} \right) =\frac{pq}{2} + h_{n-2} - \frac{q \beta_{n-2}}{2} = q \left( \frac{p-\beta_{n-2}}{2}\right) + h_{n-2} \in \la \mathbf h \bs \{q\}\ra
\end{equation}
since $p$ and $\beta_n$ have the same parity, and 
\begin{equation} \label{c22}
    q\left( \frac{p+\beta_1}{2} \right) =\frac{pq}{2} + h_{1} - \frac{p\alpha_1}{2} = p \left( \frac{q-\alpha_1}{2}\right) + h_{1} \in \la\mathbf h \bs \{p\}\ra,
\end{equation}
since $q$ and $\alpha_1$ are of same parity.

It remains to show the that $\frac{q+\alpha_{n-2}}{2}$ and $\frac{p+\beta_1}{2}$ are the smallest integers that make (\ref{c11}) and (\ref{c22}) true, respectively.

Suppose $\delta q = \gamma p + \mu_1 h_1 + ... + \mu_{n-2} h_{n-2}$, for some $\delta, \gamma, \mu_i (1 \leq i \leq n-2) \in \N$.  Not all $\mu_i$ are zero as $p$ and $q$ are coprime. Then
\begin{equation*}
 q \left( \delta - \frac{\mu_1 \beta_1 + ... + \mu_{n-2} \beta_{n-2}}{2} \right) = p \left( \gamma + \frac{\mu_1 \alpha_1 + ... + \mu_{n-2} \alpha_{n-2}}{2}\right)   
\end{equation*}

From the above equation, it follows that $p$ divides $2\delta - (\mu_1 \beta_1 + ... + \mu_{n-2} \beta_{n-2}$), so  
\begin{equation} \label{pqx}
    0 \le px = 2\delta - (\mu_1 \beta_1 + ... + \mu_{n-2} \beta_{n-2}) 
\end{equation}
for some $x \in \N$.  Now, if $q$ is odd, then none of the $\alpha_i$ are zero and hence $x>0$. 

Observe that necessarily some $\mu_j \geq 1$, so in (\ref{pqx}), 
\begin{align*}
    \delta &= \frac{px}{2} + \frac{\sum\limits_{i\neq j, 1\le i\le n}\mu_i \beta_i}{2} + \frac{p+\beta_j}{2} - \frac{p+\beta_j}{2} \\
    &= \frac{p+\beta_j}{2} + \frac{p}{2} \left( x - 1 \right) + \frac{\beta_j}{2}(\mu_j-1) + \sum_{i \neq j} \mu_i \beta_i \\
    &\geq \frac{p+\beta_j}{2} \\
    &\geq \frac{p+\beta_1}{2}
\end{align*}

We have shown that $\frac{p+\beta_1}{2}$ is the smallest positive integer such that $\left(\frac{p+\beta_1}{2}\right) p \in \la \mathbf h \bs \{p\} \ra$.

Even if $q$ is even, $x>0$ unless $\alpha_{n-2}= 0$ and $2h_{n-2} = p\beta_{n-2}$.  In that case, $T$ will still be principal as long as $p\le 2\beta_{n-2}-\beta_1$.  

Similarly, suppose $$\delta' p = \gamma' q + \nu_1 h_1 + ... + \nu_{n-2} h_{n-2},$$ for some $\delta', \gamma', \nu_i (1 \leq i \leq n-2) \in \N$, necessarily some $\nu_j \geq 1$. Isolating $p$ and $q$ as above, we obtain that $$qy = 2\delta' - (\nu_1 \alpha_1 + ... + \nu_{n-2} \alpha_{n-2})$$ for some $y \geq 1$ if $p$ is not even so that none of the $\beta_i$ are zero.

Then \begin{align*}
   \delta' &= \frac{qy}{2} + \frac{\nu_1 \alpha_1 + ... + \nu_{n-2} \alpha_{n-2}}{2} + \frac{q+\alpha_j}{2} - \frac{q+\alpha_j}{2}\\
    &\geq \frac{q+\alpha_j}{2} \\
    &\geq \frac{q+\alpha_{n-2}}{2}
\end{align*}  

Again, if $p$ happens to be even, $T$ will still be principal as long as one of the conditions (i)-(v) is statisfied.
Thus, $T$ is a principal matrix of $H$. Moreover, each $h_i$, $1 \leq i \leq n-2$ can expressed as follows: 
\begin{equation} \label{hi}
   h_i = \frac{p\alpha_i + q \beta_i}{2} = pq-\left( \frac{q-\alpha_i}{2} \right) p - \left( \frac{p-\beta_i}{2}\right) q 
\end{equation}

By hypotheses, $\alpha_i > -1$ and $\beta_i > -1$, so that $ 2 \left(\frac{q-\alpha_i}{2} \right) = q - \alpha_i < q+1$, and $2 \left( \frac{p-\beta_i}{2} \right) < p+1 $. In view of the sufficient condition to be a KW-semigroup (\ref{KWcondition}), it now follows that $H \in KW(p,q)$.

Finally, say $p$ is even, and  $2h_1 = p\alpha_1$ with  $q > 2\alpha_1-\alpha_{n-2}$.  Then $q$ is odd.  And the minimal relation on $p$ is $\alpha_1 p = 2h_1$.  Hence the principal matrix is 
 $$
        \begin{bmatrix}
    - \alpha_1 & 0 & 2 & 0 & \cdots & 0  \\
    \frac{q- \alpha_1}{2} & - \left(\frac{p+\beta_1}{2} \right) & 1 & 0 & \cdots & 0  \\
    \alpha_1 & \beta_1 & -2 & 0 & \cdots &  0 \\
    \alpha_2 & \beta_2 & 0 & -2 & \cdots  & 0 \\
    \vdots & \vdots & \vdots & \vdots & \ddots & \vdots \\
    \alpha_{n-2} & \beta_{n-2} & 0 & 0 & \cdots & -2 
    \end{bmatrix}$$

    The case when $q$ is even can be proved similarly. 
\end{proof}

Consequently, we obtain a characterization of the principal matrix of any $H\in KW(p,q)$ when $3<p<q$ are prime.  This is Corollary \ref{maincor}. Recall from the Introduction:
\begin{corollary*}[\ref{maincor}]
    Let $p$ be odd and $q>p$ be relatively prime to $p$.   Then the principal matrix of  a numerical semigroup $A =( p , q, h_1, \ldots, h_{n-2})\in KW((p,q))$  is of the form 
$$P(A)=  T(\alpha, \beta)=
    \begin{bmatrix}
    - \left(\frac{q + \alpha_{n-2}}{2} \right) & \frac{p-\beta_{n-2}}{2} & 0 & 0 & \cdots & 0 & 1 \\
    \frac{q- \alpha_1}{2} & - \left(\frac{p+\beta_1}{2} \right) & 1 & 0 & \cdots & 0 & 0 \\
    \alpha_1 & \beta_1 & -2 & 0 & \cdots & 0 & 0 \\
    \alpha_2 & \beta_2 & 0 & -2 & \cdots & 0 & 0 \\
    \vdots & \vdots & \vdots & \vdots & \ddots & \vdots & \vdots \\
    \alpha_{n-2} & \beta_{n-2} & 0 & 0 & \cdots & 0 & -2 
    \end{bmatrix},$$ for some odd positive integers $\alpha_i, \beta_i$ satisfying $q> \alpha_1 \ge \ldots \ge \alpha_{n-2}\ge 0$ and $0< \beta _1\le \cdots \leq \beta_{n-2}<p$.  Conversely, any $n\times n$ matrix of the form $P(A)= T(\alpha, \beta)$ above is a principal matrix of a $KW(p,q)$ numerical semigroup. 
\end{corollary*}
\begin{proof}
Note that $h_i = pq-px_i-qy_i$, for all $1 \leq i \leq n-2$, with $0<2x_i \leq q, 0 < 2y_i \leq p$. 
     Since $H \in KW(p,q)$, it satisfies the hypotheses of Theorem $\ref{main}$. Next, since $p\geq 3$ is prime, it must be odd so necessarily $y_1 \neq p/2$. Thus, we get that the principal matrix of $H$ is of the form $T(\alpha, \beta)$ from equation $\eqref{principalmatrix}$.

     The converse is a direct consequence of Theorem $\ref{main}$.
\end{proof}

\begin{remark}\label{maxembdim}
The maximal embedding dimension for a numerical semigroup $H$ in $KW(p,q)$ is precisely $2+\lfloor \frac{p}{2}\rfloor$.  This maximum is achieved when the set of minimal generators of $H$ is a part of an arithmetic sequence $(p, p+d=q, \ldots, q+(p-1)d)$.  However, the building of this extremal semigroup is via the gaps $h_i = q+(p-1)d = pq-p\lfloor \frac{q}{2} \rfloor-q$ and continue.   
For example, with $p=7, q=11$ we see that $h_1 = 77-35-11= 31, h_2 = 27, h_3= 23$ which is contained in $(7,11,15,19, 23, 27,31)$. 
\end{remark}

\section{Auxiliary Results}

Let $H \in KW(p,q)$ be a numerical semigroup of embedding dimension $n\geq 3$.  Thus,
$H$ is minimally generated by the set $\{p,q,h_1,...,h_{n-2}\}$ for some relatively prime positive integers $p$ and $q$, and $h_i = pq-px_i-qy_i$, $1 \le i \le n-2$. Here, $q> x_{n-2}> ... > x_{1}$  and $p> y_{1} > ... > y_{n-2}$ are positive integers with the same parity with $p$ and $q$ respectively.  Let $k$ be a field. Set $S=k[u, v, u_1,...,u_{n-2}]$ to be a polynomial ring and $\phi : S \to k[t]$ to be the map defined by $u \mapsto t^p, v \mapsto t^q, u_{i} \mapsto t^{h_i}, 1 \leq i \leq n-2$. Then the semigroup ring $k[H]=k[t^p,t^q,t^{h_1},...,t^{h_{n-2}}]$ is the image of $\phi$ and it is isomorphic to $S/I_H$ where $I_H = \ker \phi$, called the ideal of $H$, is the prime ideal defining the affine monomial curve whose coordinate ring is $k[H]$.

\subsection{A Classification of $I_H$}
Using results from the appendix in \cite{KUNZ2017397},  we can express the defining ideal of the semigroup ring $k[H]$ where $H \in KW(p,q)$ for any embedding dimension $n \geq 3$ as the sum of determinantal ideals.  

\begin{theorem}
    Let $H = \la p,q, h_1,...,h_{n-2} \ra \in KW(p,q)$, so that $h_i = pq-x_ip-y_iq$, $1 \leq i \leq n-2$ with positive integers $x_1 < x_2 < ... < x_{n-2}$, $y_1 > y_2 > ... > y_{n-2}$, each satisfying \eqref{KWcondition}: $2x_i \leq q, 2y_i \leq p$. Then the defining ideal $I_H$ of the semigroup ring $k[H]$ is the ideal of the $2 \times 2$ minors of the following $2 \times 3$ matrices: 

    $$ 
    A_{ij} = 
    \begin{bmatrix}
        u_i & u^{q-2x_j}v^{p-y_i-y_j} & u_j \\
        u^{x_j-x_i} & u_j & v^{y_i-y_j}
    \end{bmatrix}, \quad \quad     1 \leq i < j \leq n-2
    $$$$
     \quad
    B =
    \begin{bmatrix}
        v^{y_{n-2}} & u^{q-x_1-x_{n-2}} & u_1 \\
        u^{x_1} & u_{n-2} & v^{p-y_1-y_{n-2}}
    \end{bmatrix}
    $$$$
     C=
    \begin{bmatrix}
        u^{q-x_1-x_{n-2}}v^{p-2y_1} & u_1 & u_{n-2} \\
        u_1 & u^{x_{n-2}-x_1} & v^{y_1-y_{n-2}}
    \end{bmatrix} 
$$

\end{theorem}

\begin{proof}
    According to the theorem in Appendix A. of \cite{KUNZ2017397}, the defining ideal of $H$ consists of the following binomials:
    \begin{align*}
        f_{ij} &= u_iu_j-u^{q-x_i-x_j}y^{p-y_i-y_j}, \quad 1 \leq i \leq j \leq n-2 \\
        g_i &= v^{y_i-y_{i+1}}u_i-u^{x_{i+1}-x_i}u_{i+1}, \quad 1 \leq i \leq n-3 \\
        h_1 &= v^{p-y_1}-u^{x_1}u_1 \\
        h_2 &= v^{y_{n-2}}u_{n-2}-u^{q-x_{n-2}}
    \end{align*}
    Let $M^l$ denote column $l$ of a matrix $M$. Observe that
    \begin{align*}
        f_{ij} &= \det
        \begin{bmatrix}
            A_{ij}^1 & A_{ij}^2
        \end{bmatrix}, \quad 1 \leq i < j \leq n-2, \\
        f_{jj} &= \det 
        \begin{bmatrix}
            A_{ij}^2 & A_{ij}^3
        \end{bmatrix}, \quad 2 \leq j \leq n-2, \\
        f_{11} &= \det 
        \begin{bmatrix}
           C^1 & C^2 
        \end{bmatrix}, \\
        g_i &= \det
        \begin{bmatrix}
            A_{ij}^1 & A_{ij}^3
        \end{bmatrix}, \quad 1 \leq i \leq n-3, \: \:  j=i+1, \\
        h_1 &= \det 
        \begin{bmatrix}
            B^1 & B^3
        \end{bmatrix}, \\
        h_2 &= \det
        \begin{bmatrix}
            B^1 & B^2
        \end{bmatrix}
    \end{align*}
\end{proof}

\subsection{Minimal Free Resolutions}

\newcommand{\al}{\alpha_1}
\newcommand{\all}{\alpha_2}
\newcommand{\be}{\beta_1}
\newcommand{\bee}{\beta_2}
\newcommand{\A}{\alpha}
\newcommand{\B}{\beta}

In this section, we explicitly describe minimal free resolutions of the semigroup ring $k[H]$ where $H$ is a numerical semigroup of small embedding dimension $3$ or $4$ in the class $KW(p,q)$, and give their Betti numbers. We express the maps in the resolution in terms of the $x_i, y_i$ where $px_i+qy_i$ are the complements of the gaps $h_i$ of $\la p,q \ra$, i.e., $h_i = pq - (px_i + qy_i)$ for $1 \leq i \leq n-2$.  \\

\subsubsection{Embedding Dimension of $H$ is $3$} In this instance, we have $n=3$, so \\ $H = \langle p,q,h_1 \rangle \in KW(p,q)$, $h_1= pq-px_1-qy_1$. Here, $x_1,y_1 \in \N$ with $(x_1,y_1) \leq (q/2,p/2)$. Furthermore, $k[H] = k[t^p,t^q,t^{h_1}]$ is the semigroup ring associated to $H$, and $\phi : S \rightarrow R$ is the map defined by $u \mapsto t^p, v \mapsto t^q, u_1 \mapsto t^{h_1}$. Following is the resolution of $k[H]$: $$0 \to S^2 \xrightarrow{\begin{pmatrix}
    u_1 & -u^{q-2x_1}\\
    -u^{x_1} & v^{y_1}\\
    v^{p-2y_1} & -u_1
\end{pmatrix}} S^{3} \xrightarrow{\begin{pmatrix}
    v^{p-y_1}-u^{x_1} & u^{q-2x_1}v^{p-2y_1} -u_1^2 & u^{q-x_1} -v^{y_1}u_1
\end{pmatrix}} S^1 $$

\subsubsection{Embedding Dimension of $H$ is $4$}
In this setting, $H = \langle p,q,h_1,h_2 \rangle \in KW(p,q)$, $h_i= pq-px_i-qy_i, i=1,2$. Here, $x_1<x_2\leq q/2$, $y_2<y_1 \leq p/2$ are in $\N$. Furthermore, $k[H] = k[t^p,t^q,t^{h_1},t^{h_2}]$ is the semigroup ring associated to $H$. Following is the resolution of $R/I$, where $I$ is the relation ideal of $k[H]$: $$0 \to S^3 \xrightarrow{A_1} S^8 \xrightarrow{A_2} S^6 \xrightarrow{A_3} S^1$$
The maps are given by matrices below. \vspace{.2in}\\
$$A_1 = 
\begin{bmatrix}
u_2 & -u^{q-x_2}v^{y_1-y_2} & 0 \\
-u_1 & u^{q-x_1-x_2} & 0 \\
u^{x_1}v^{p-2y_1} & -x_1 & 0 \\
0 & v^{y_2} & u^{x_2-x_1} \\
-v^{p-y_1-y_2} & u_2 & 0 \\
-u^{2x_1-x_2} & 0 & -v^{y_1-y_2} \\
u^{x_2-x_2} u_1 & 0 & u_2 \\
0 & v^{2y_2-y_1} u_2 & -u_1
\end{bmatrix},$$ 
\vspace{.3in}\\
$A_2 =$\\
\resizebox{\linewidth}{!}{%
$\begin{bmatrix}
-u^{x_1} & -u^{2x_1-x_2} v^{p-y_1-y_2} & 0 & 0 & 0 & -u_2 & -v^{p-y_1-y_2} & 0 \\
0 & 0 & 0 & -u_2 & 0 & u_1 & u^{x_2-x_1} & -v^{\frac{y_1-y_2}{2}} \\
0 & 0 & v^{y_2} & u_1 & u^{x_2-x_1}v^{2y_2-y_1} & 0 & 0 & u^{x_2-x_1} \\
-v^{p+y_2-2y_1} & -u^{2x_1-x_2} u_1 & u^{q+x_1-2x_2} & -u^{q-2x_2} v^{p-y_1-y_2} & -v^{2y_2-y_1}u_2 & -u^{q-x_1-x_2}v^{p-2y_1} & -u_1 & -u_2 \\
-u^{x_2-x_1} v^{p-2y_1} & -v^{p-y_1-y_2} & u_2 & 0 & u_1 & 0 & 0 & 0 \\
-u_1 & -u_2 & u^{q-2x_2} v^{y_1-y_2} & 0 & u^{q-x_1-x_2} & 0 & 0 & 0 
\end{bmatrix}
$}
\vspace{.3in}\\
$A_3=$ \\
\resizebox{\linewidth}{!}{%
$ 
\begin{bmatrix}
u^{q-2x_1}v^{p-2y_1} - u_1^2, & u^{q-x_1-x_2}v^{p-y_1-y_2} - u_1u_2, & u^{q-2x_2}v^{p-2y_2} - u_2^2, & v^{y_1-y_2}u_1 - u^{x_2-x_1}u_2, & v^{y_2}u_2 - u^{q-x_2}, & u^{x_1}u_1 - v^{p-y_1}
\end{bmatrix}
$}

\vspace{.2in}

\section{Generalizing KW Type Monoids to Dimension Three}

In this section, we will expand this class $KW(p,q)$. Let $p$ and $q$ be two coprime numbers with $1 < p < q$. Let 
 $w=r_1p+r_2q$ for some  $r_1, r_2\in \N$. Then the semigroup $S= \la sp,sq,r_1p+r_2q \ra$ with $\gcd(s,w)=1$, is again symmetric and with Frobenius number $s(pq-p-q)+w(s-1)$.  Consider  the class $R(p,q,r_1,r_2,s)$ of numerical semigroups obtained by adding to $S$ gaps corresponding to points $(a,b,c) \in \N^3$ falling on a lattice path under the plane $G_0: spx+sqy+ wz = s(pq+w)-sp-sq-w$. 

From here on, $S$ is the numerical semigroup $\la sp, sq, w\ra$ where $p<q$ are relatively prime and $s,w$ are relatively prime and $w =r_1p+r_2q, r_i\ge 1$. 
\begin{figure}[h]
\caption{Gaps of $\langle 15, 21, 17 \rangle$ are in 1-1 correspondence with lattice points under the plane $15x+21y+17z=103$.}
\includegraphics[scale=0.85]{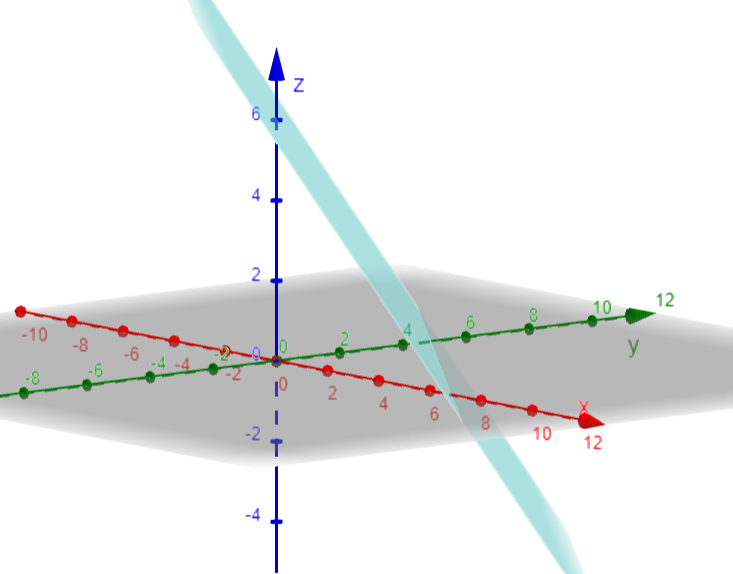}
\end{figure}

\begin{proposition}\label{!GapExp}
 Every gap of $S$ can be written as $s(pq+w) - sp(a+1)-sp(b+1) -w(c+1)$ for some $(a,b,c) \in \N^3$. We may always take $c\le s-1$. In particular, each gap of $S$ is of the form $s(pq+w) - sp(a+1)-sp(b+1) -w(c+1)$  with $(a,b,c)<(q,p,s)$ alphabetically, and corresponds to a unique lattice point $(a,b,c)$ below the hyperplane $$G_0: spx+sqy+ wz = s(pq+w)-sp-sq-w.$$
\end{proposition}
\begin{proof}
   Let $t$ be a gap of $S$. Since $S$ is symmetric, $F(S)-t\in S$ so that $F(S) -t = spx+sqy+wz$ for some $(x,y,z) \in \N^3$. In other words, 

   \begin{equation}
       t= s(pq+w) -sp(x+1)-sq(y+1)-w(z+1)
   \end{equation}

Now assume $x<q, y< p, z<s$.   Suppose 
 \begin{multline*}
     t= s(pq+w) -sp(x+1)-sq(y+1)-w(z+1) \\
     = s(pq+w) -sp(x'+1)-sq(y'+1)-w(z'+1)
 \end{multline*}
 for some $(x,y,z),(x',y',z') \in \N^3$. Then $$s(px+qy-px'-qy') = w(z'-z)$$
Since $(s,w)=1$, $s \mid z'-z$, i.e.,   $z'=z+ms$ for some $m \in \Z$. But $0 \leq z,z' \leq s-1$ so $m=0$. 

Therefore,  $$p(x-x')=q(y'-y).$$  Since $p$ and $q$ are relatively prime, we get  $x=x'+mq$ for some $m \in \Z$. But $0 \leq x,x' \leq q-1$ so $m=0$. 

Thus, we must have $x=x', y=y'$, and $z=z'$.

\end{proof}

\begin{definition}
We denote by $KW(p,q,r_1,r_2,s)$ the numerical semigroups $H$ in $R(p,q,r_1,r_2,s)$   generated by $\{ sp,sq, w= r_1p+r_2q, h_1, \ldots, h_{n-3}\}$ where $h_i = s(pq+w)-(x_i+1)sp-(y_i+1)sq- (z_i+1)w$, where $x_i\le {\frac {q-4}{2}}, y_i\le {\frac {p-4}{2}}$ and $z_i \le s-2$. 
\end{definition}

\begin{figure}[h]
\caption{$H$ obtained by adding $55,51,63$ to $S$}
\includegraphics[scale=0.6]{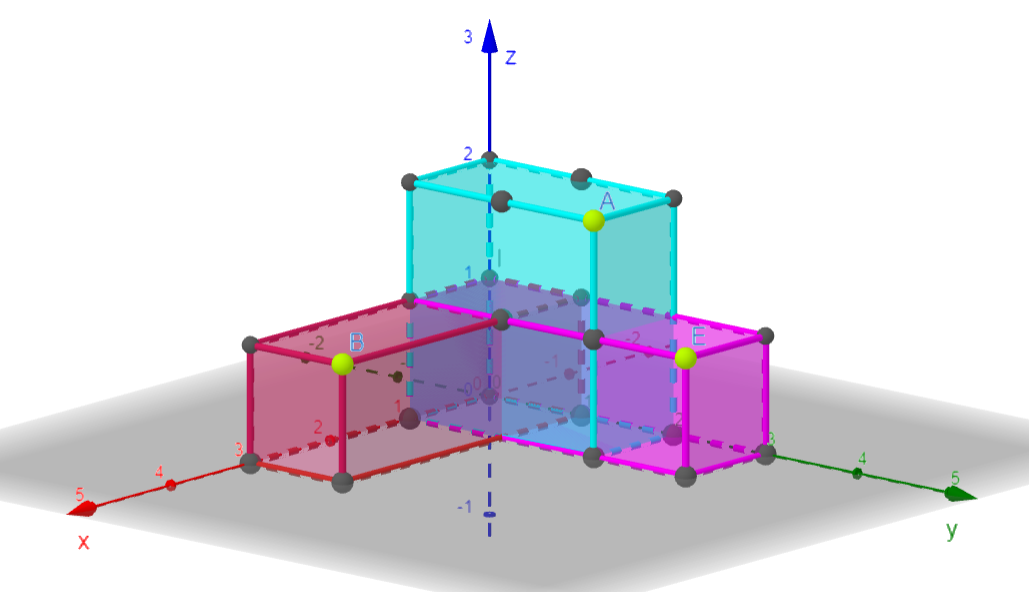}
\end{figure}

Set $\Gamma(x,y,z) = F(S)-xsp-ysq-zw$ for $(x,y,z) \in \N^3$.The Ap\'ery set of a numerical semigroup $\mathbf a$ with respect to an element $a \in \mathbf a$, named in honour of R. Ap\'ery, is $Ap(\mathbf a, a) = \{b \in \mathbf a \mid b-a \not\in \mathbf a\}$. Let $H= \la sp,sq,w,h_2,...,h_{n-2} \ra \in KW(p,q,r_1,r_2,s)$. We give a characterization of the elements $h_2,...,h_{n-2}$ in terms os the Ap\'ery sets of $H$.  To understand the motivation, we go through the following geometric illustration.

    \begin{example}

    Let $S=\la sp=36,sq=44,w=29 \ra$. Obtain $H$ by adding the gaps corresponding to a lattice point within one of the cuboids in Figure 2. Here, $H= \la 36,44,29,206,221,222  \ra \in KW(9,11,2,1,4)$, where $h_2=221,h_3=222,h_4=206$ correspond to the lattice points $A=(1,2,2), B=(3,1,1)$ and $E=(1,3,1)$, respectively. 

    First, we compute $\Ap(H,sp=36)$. Fix $z=c$, with $c \leq z_1$, the maximum height of any cuboid in the image. This plane passes through at least one of the three cuboids, say $C_i$ which is obtained from $h_i$. For each $b \leq b_i$, there is a line $y=b,z=c$ on the plane $z=c$ such that each point $(a_j,b,c), a_j \leq a_2$ is in $H$. Moreover, whenever $i,j \leq 2$, we have $$\G(a_i,b,c) \equiv \G(a_j,b,c)\mod sp$$ 
    
    Out of all these points on the line $z=c,y=b$, the one with the largest $x-$coordinate, say $P_l$, corresponds to the smallest gap congruent to $sw-sqb-wc$ modulo $sp$. The equivalence classes are distinct for each distinct $b$ and $c$, so $\G(P_l) \in \Ap(H, sp)$. Visually, $Ap(H,sp)\bs S$ is contained in the set of gaps corresponding to the lattice points seen in Figure 3.
\begin{figure}
\centering
\includegraphics[width=0.50\textwidth]{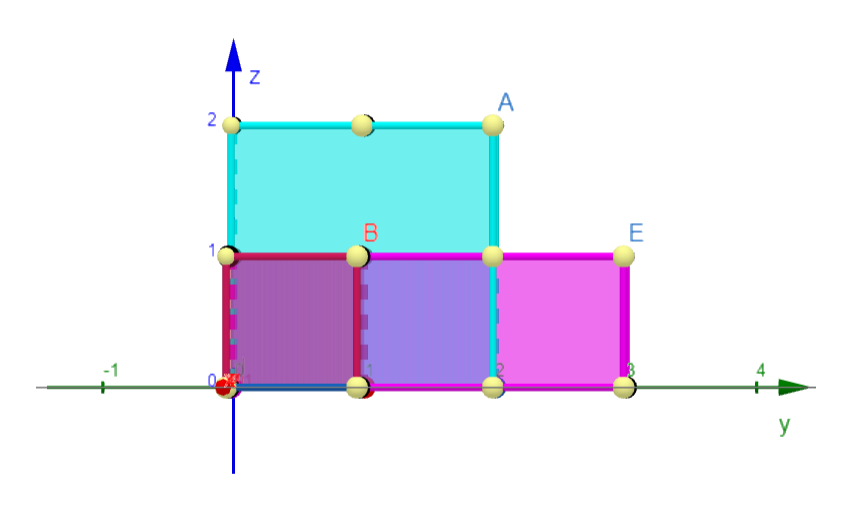}
\caption{$Ap(H,sp)\bs S$}
\end{figure}

    Similarly, we can find elements in $\Ap(H,sq=44)$ and $\Ap(H,w=29)$. $Ap(H,sq)\bs S$ and $Ap(H,w) \bs S$ can be visualized in Figures 4 and 5 respectively. One can see that the only lattice points common to all three figures $3,4$ and $5$ are exactly those corresponding to $h_2,h_3$ and $h_4$.
     \end{example}

\begin{figure}[!htb]
   \begin{minipage}{0.48\textwidth}
     \centering
     \includegraphics[width=.7\linewidth]{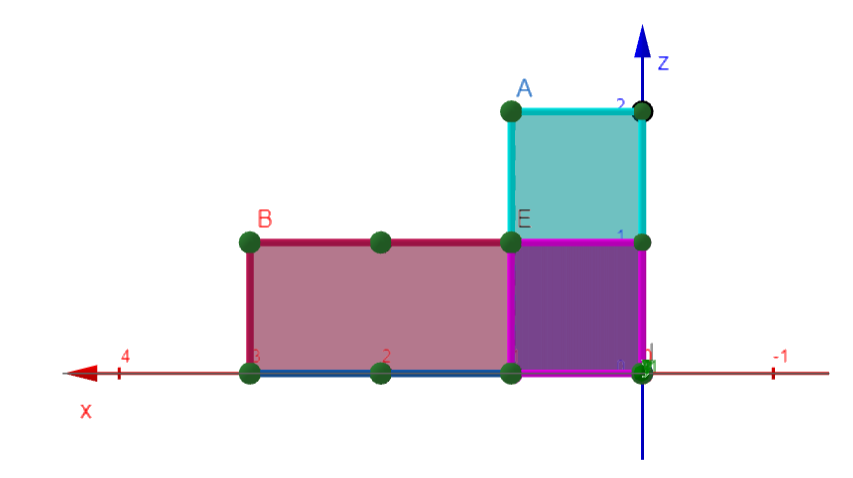}
     \caption{$Ap(H,sq)\bs S$}
   \end{minipage}
   \begin{minipage}{0.48\textwidth}
     \centering
     \includegraphics[width=.7\linewidth]{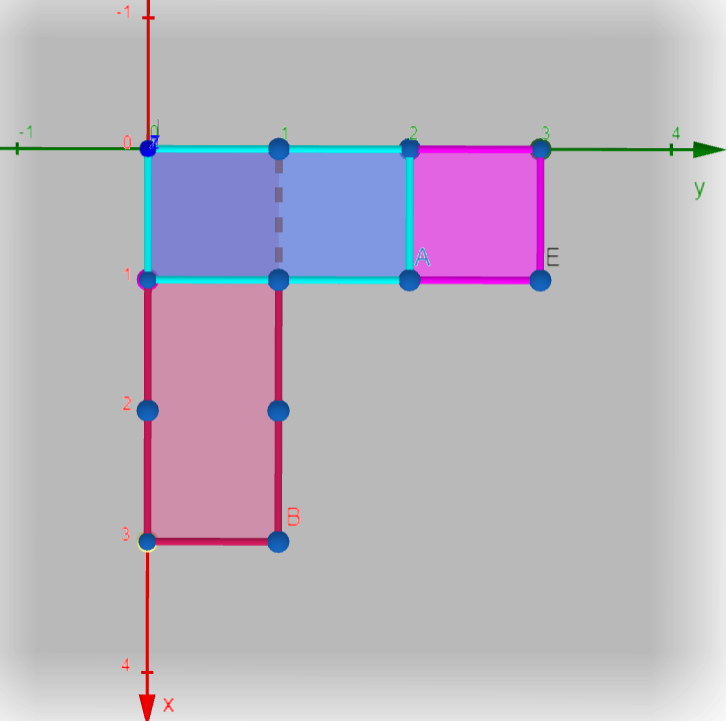}
     \caption{$Ap(H,w)\bs S$}
   \end{minipage}
\end{figure}

  This gives rise to the following result:

\begin{proposition}
    Let $H \in KW(p,q,r_1,r_2,s)$. Then $\{\text{minimal generators of } H\} \bs \{sp,sq,w\} = \Ap(H,sp) \cap \Ap(H,sq) \cap \Ap(H,w)$.
\end{proposition}

\begin{proof}

    Denote by $\mathbf h$ the minimal generating set of $H$. Let $h_i\in \mathbf h$.  If $h_i-sp \in H$, then $h_i = \sum_{j}r_jh_j + a sp+bsq+cw $.  If $r_i =0$, $h_i$ is not a minimal generator. If $r_i \ge 1$, then we have $\sum_{j\neq i}r_jh_j +(r_i-1)h_i (a+1)sp+bsq+cw =0$ which is impossible.  Hence $h_i-sp\notin H$ and hence $h_i \in \Ap(H,sp)$.  Similarly, $h_i \in \Ap(H, sq)$ and $h_i \in \Ap(H,w)$. Thus, $\mathbf h \bs \{p,sq,w\} \subset \Ap(H,sp) \cap \Ap(H,sq) \cap \Ap(H,w)$ since $x \notin \Ap(H,x)$ for any $x \in H$.

    If $h \in \Ap(H,sp) \cap \Ap(H,sq) \cap \Ap(H,w)$ then $h = \sum_{i=1}^{n-2} t_i h_i$ for $h-sp, h-sq, h-w $ cannot be in $H$.  Further since $2h_i \in \la sp,sq,w \ra$ for all $i$, $t_i \le 1$.  
If two of the $t_i$ are not zero, say, $i = 1,2$, then we have $h = 2h_1-sp(x_2-x_1) -sq(y_2-y_1)- w(z_2-z_1) + \sum_{i =3}^{n-2} h_i$. 
By the fact that $h_i \in H$ are minimal generators, we have that one of $(x_2-x_1), (y_2-y_1), (z_2-z_1)$ must be negative. So, one of $h-sp,h-sq,h-w$ must be in $H$ which is impossible. so, only one of $t_i\neq 0$.  Hence, $h= h_i $ for some $i$ and so we get the other inclusion: $\mathbf h \bs \{sp,sq,w\} = \Ap(H,sp) \cap \Ap(H,sq) \cap \Ap(H,w)$. 
\end{proof}

The proposition above establishes a close relationship between the lattice path of a numerical semigroup $H \in KW(p,q,r_1,r_2,s)$ and its Ap\'ery set. In fact, 
 the lattice path of $H$ can be exclusively determined by its Ap\'ery sets $\Ap(H,sp)$, $\Ap(H,sq)$, and $\Ap(H,w)$. Consequently, this observation could facilitate deriving an upper bound for the embedding dimension of any $H \in KW(p,q,r_1,r_2,s)$, as in Remark \ref{maxembdim}.

Kunz and Waldi show that for  $H \in KW(p,q)$, the type of the semigroup ring $k[H]$ equals $e(H)-1$. 
Contrasting this with the dimension three, consider $H\in R(5,7,2,1,3)$ obtained by adding a gap $h$ of $\la 15,21,17\ra$. In the following table. the first column consists of gaps $h$ such that $H= \la 15,21,17,h \ra$, with embedding dimension $4$ has the corresponding type in column $2$.  
\vspace{.1in}\\
\begin{center}
    \begin{tabular}{|c|c|}
    \hline
   Choice of $h$  & $t(H)$\\
   \hline
    50,56,65,67, 71,73,82,86,88,103 & 3  \\
    35, 52 & 4 \\
    \hline
\end{tabular}
\end{center}
\vspace{.3in}

However, we show that for a class of the  $h\in H$ in $KW(p,q,r_1,r_2,s)$, we prove that the type is always $3$. 

\begin{example}
    Consider Figure $6$ below. Take $H=\la 36,44,29,221 \ra \in KW(9,11,2,1,4)$ where $221$ corresponds to the lattice point $A=(1,2,2)$. Then $PF(H)=\{271,316,331\}$ where $331,271,$ and $316$ correspond to the points $I=(2,0,0), J=(0,3,0),$ and $K=(0,0,3)$, respectively. 
\end{example}

\begin{figure}[h]
\centering
\includegraphics[width=0.5\textwidth]{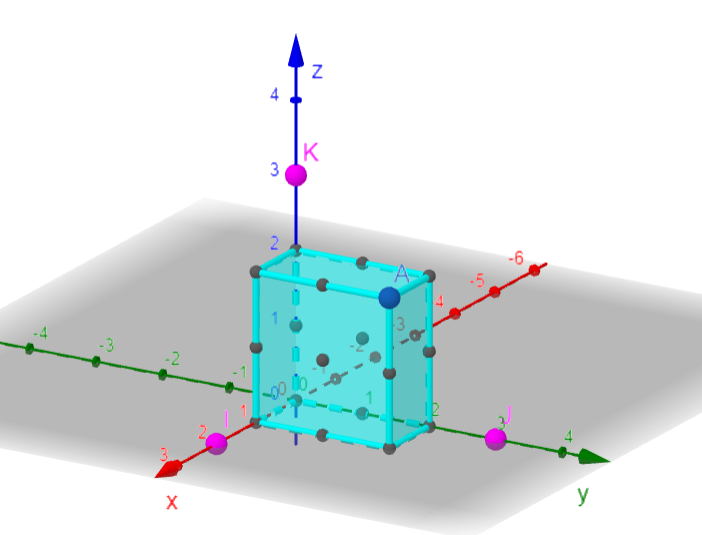}
\caption{$PF(\la 36,44,29,221 \ra)$}
\end{figure}

\begin{theorem}
    Let $H= \la sp,sq,w,h \ra \in  KW(p,q,r_1,r_2,s)$ with $h=F(S)-spx-sqy-wz$.  If $p_z:=F(S)-(z+1)w \not\in H$, then $t(H)=3$.
\end{theorem}
\begin{proof}
    We will show that $$PF(H) = \{p_x:=F(S)-(x+1)sp, p_y:=F(S)-(y+1)sq, p_z\}.$$ 
    \begin{itemize}
    \item[Claim 1.] $p_x, p_y \not\in H$. Suppose first that $p_x \in H$. Then $p_x= \alpha sp + \beta sq + \gamma w + \delta h$ for some $\alpha, \beta, \gamma, \delta \in N$.  By the choice of $h$ to be in  Kunz-Waldi class, $2h \in S$. $pX$ is a gap of $S$.  So, $\delta = 1$.  So, we get $(\alpha+1)sp+(\beta-y)sq+(\gamma-z)w = 0$.  This means either $\beta<y$ or $\gamma<z$.  If indeed $\gamma<z$, then $(z-\gamma)w = (\alpha+1)sp+(\beta-y)sq$.  So, $s$ must divide $z-\gamma$ but this is impossible since $z<s$.  Now if $\gamma\ge z$, $0\le (c-z)w= ((y-\beta)q-(\alpha+1)p)s$.  So, there exists $u \in \N$ such that $(y-\beta)q-(\alpha+1)p = w u= r_1pu+r_2qu \implies  (y-\beta-r_2u)q = (\alpha+1+r_1u)p >0 \implies y>p$ a contradiction. So, $p_x\notin H$.  The argument for $p_y \not\in H$ is similar.

        \item[Claim 2.] $p_i + sp, p_i+sq, p_i+w, p_i+h \in H$ for all $1 \leq i \leq 3$.\\
 To see that $p_1+sp \in H$, we have the following
            \begin{align*}
                p_1 + sp &= F(S) - (x+1) sp + sp \\
                &= F(S) -xsp \\
                &= h +ysq+zw \in H
            \end{align*}
      Next, to see that $p_1+sq \in H$, we have     
\begin{align*}
                p_1 +sq &= F(S)-(x+1)sp + sq \\
                &= spq+sw -sp-sq-w -(x+1)sp +sq\\
                &=sp(q-x-2) +w(s-1) \in H
            \end{align*}
            since $x \leq \lfloor q/2 \rfloor -2 \implies q-x-2 \geq 0$.
 Thirdly, for $p_1+w \in H$, we see
            \begin{align*}
                p_1+w &=  F(S) - (x+1) sp + w \\
                &= spq+sw -sp-sq-w -(x+1)sp +w \\
                &= sp(q-x-2) +s(r_1p+r_2q)-sq \\
                &= sp(q-x-2+r_1)+sq(r_2-1) \in H
            \end{align*}
            since $r_2 \geq 1$.

Similarly, $p_2 + sp = sq(p-y-2)+w(s-1) \in H$ since $y \leq \lfloor p/2 \rfloor -2 \implies p-y-2 \geq 0$, $p_2+sq = h+xsp+zw \in H$, and $p_2 +w = sq(p-y-2+r_2) + sp(r_1-1) \in H$ as $r_1 \geq 1$.

Next, $p_3 + sp = w(s-z-2) +sq(p-1) \in H$ as $z \leq s-2$, $p_3 + sq = w(s-z-2) +sp(q-1) \in H$, and  $p_3+w = h + xsp + ysq \in H$.

It remains to show that $p_i +h \in H$ for all $1 \leq i \leq 3$. Well, recall that  $x \leq \frac{q-4}{2}, y \leq \frac{p-4}{2},$ and $z \leq s-2$ so that
\begin{align*}
    p_1+h &= sp(q-2x-3) +sq(p-y-2)+w(2s-z-2) \in H,
\end{align*}
\begin{align*}
    p_2+h &= sp(q-x-2) +sq(p-2y-3)+w(2s-z-2) \in H,
\end{align*}
and \begin{align*}
    p_3+h &= sp(q-x-2) +sq(p-y-2)+w(2s-2z-2) \in H.
\end{align*}

\item[Claim 3.] No other gap of $H$ is in $PF(H)$. Any gap of $H$ is of the form $g=F(S) -spa-sqb-wc$ for some $(a,b,c) \in \N^3$. Suppose $g$ is different from $p_1,p_2,p_3$.  It is not of the form $F(S)-(x+1+i)sp$ for $i> 0$ because $F(S)-(x+1+i)sp + sp = p_1 - (i-1)sp \in H \implies p_1 \in H$, a contradiction to Claim 1. Similarly, $g$ is not of the form $F(S)-(y+1+i)sq$ or $F(S)-(z+1+i)w$ for $i\geq 0$. Further, one of the following must happen: $a>x, b>y, c>w$. Otherwise, if $a\leq x, b \leq y, c \leq z$, then $ g = h + (x-a)sp + (y-b)sq +(z-c)w \in H$.

Suppose $g=F(S) -spa-sqb-wc \in PF(H)$ with $a=x+j$ for some $j \in \N \bs \{0\}$, and one of $b, c$ is non-zero. Then $p_1 = g+sp(j-1) +sqb+wc \in H.$ If $b=y+j$ for some $j \geq 1$, then $p_2=g+spa +sq(j-1)+wc \in H.$ Finally, if $c=z+j$ for some $j \geq 1$, then $p_3=g+spa +sqb+w(j-1) \in H.$ Each of these cases is a contradiction since $p_1,p_2,p_3 \not\in H$. So $g$ cannot be different from $p_1,p_2$ or $p_3$.

Thus, $PF(H) = \{p_1,p_2,p_3\}$ proving that the type of $H$ is $3$.
    \end{itemize}
\end{proof}

The condition $p_3 \not\in H$ is necessary as seen in the example below.

\begin{example}
    Take $H=\la 27,33,29, 152\ra \in KW(9,11,2,1,3)$. Here, $152=s(pq+w)-(x+1)sp-(y+1)sqy-(z+1)w$ for $(x,y,z)=(3,1,1)$ but $t(H)=4$. In this case, $p_3 =237 \in H$ and $PF(H)=\{p_1,p_2,183,204\}$. 
\end{example}

Further, the theorem still may not hold if $H \in R(p,q,r_1,r_2,s) \bs KW(p,q,r_1,r_2,s)$.
\begin{example}
    Take $H=\la 15,21,17, 37\ra \in R(5,7,2,1,3)$. Here, $37=s(pq+w)-(x+1)sp-(y+1)sqy-(z+1)w$ for $(x,y,z)=(3,1,0)$, so that $3 > \frac{7-4}{2}=1.5$, and $t(H)=4$. 
\end{example}

With this in mind, we ask the following questions:
\begin{question}
What are the conditions on the gaps $h_i$ that will yield an explicit formula for the type of $H$? 
\end{question}
\begin{question}
Let $h_i =spx+sqy+ wz = s(pq+w)-spx_i-sqy_i-wz_i $, where $x_i< {\frac{q}{2}}, y_i< {\frac{p}{2}}, z<s$.  
In particular, what is the type of  $H =( sp,sq,w,h_i,1\le i\le n-2) \in KW(p,q,s,w)$ ? 
\end{question}

Theorem $\ref{KWmain}$ gives the number of generators for the relation ideal of any $H \in KW(p,q)$. Consider the following tables, In the left one, the first column consists of gaps $h$ such that $H= \la 15,21,17,h \ra$, with embedding dimension $4$, and the second column has the corresponding number of minimal generators for its relation ideal. In the next table, the first column consists of $h'$ such that $H'=\la 15,21,17,73,h' \ra$, and the second column is $\mu(I_{H'})$, the number of minimal generators of $I_{H'}$.\\
\begin{tiny}
\begin{center}
\begin{tabular}{|c|c|}
\hline
   Choice of $h$  & $\mu(I_H)$\\
   \hline
   35, 50, 52, 56, 65, 67, 71, 73, 82, 86, 88 & 6 \\
    103 & 7 \\
    \hline
\end{tabular}
\quad \quad \quad 
\begin{tabular}{|c|c|}
\hline 
   Choice of $h'$  & $\mu(I_H)$\\
   \hline
   67 & 9\\
  50,  71, 82, 86 & 10 \\
   65 & 11 \\
    \hline
\end{tabular}
    \end{center}
\end{tiny}
\vspace{.3in}

This prompts the question:

\begin{question}
What is the number of minimal generators for the relation ideal of $H \in KW(p,q,r_1,r_2,s)$?
\end{question}

Even when the embedding dimension of $H$ is $4$, we may have many different numbers of generators for its ideal as seen in the table above.

\end{document}